\documentclass[10pt,oneside]{amsart}

\usepackage{amssymb, amsmath, amsthm, amsfonts}
\usepackage{mathrsfs,comment, mathtools}
\usepackage{stmaryrd}
\input{xypic}
\usepackage{graphicx}
\usepackage{geometry}
\usepackage{float}
\usepackage[all]{xy}

\usepackage[bookmarks,bookmarksopen,bookmarksdepth=4]{hyperref}
\usepackage{url}
\usepackage{enumerate}
\usepackage{color}

\usepackage{verbatim}
\usepackage[T1]{fontenc}
\usepackage{todonotes}

\hypersetup{%
  bookmarksnumbered=true,%
  colorlinks=true,%
  linkcolor=blue,%
  citecolor=blue,%
  filecolor=blue,%
  menucolor=blue,%
  urlcolor=blue,%
  pdfnewwindow=true,%
  pdfstartview=FitBH}

\def\@url#1{{\tt\def~{\lower3.5pt\hbox{\char'176}}\def\_{\char'137}#1}}

\newtheorem{theorem}{Theorem}[section]
\newtheorem{corollary}[theorem]{Corollary}
\newtheorem{proposition}[theorem]{Proposition}
\newtheorem{lemma}[theorem]{Lemma}

\theoremstyle{definition}
\newtheorem{definition}[theorem]{Definition}
\newtheorem{example}[theorem]{Example}

\newtheorem{remark}[theorem]{Remark}

\newtheorem{construction}[theorem]{Construction}

\theoremstyle{plain}

\newtheorem{xxthm}{Theorem}
\newtheorem{xxprop}[xxthm]{Proposition}
\newtheorem{xxlem}[xxthm]{Lemma}

\newcommand{\Mdef}[2]{\newcommand{#1}{\relax \ifmmode #2 \else $#2$\fi}}
\Mdef{\bA}{\mathbb{A}}
\Mdef{\bB}{\mathbb{B}}
\Mdef{\bC}{\mathbb{C}}
\Mdef{\bD}{\mathbb{D}}
\Mdef{\bE}{\mathbb{E}}
\Mdef{\bF}{\mathbb{F}}
\Mdef{\bG}{\mathbb{G}}
\Mdef{\bH}{\mathbb{H}}
\Mdef{\bI}{\mathbb{I}}
\Mdef{\bJ}{\mathbb{J}}
\Mdef{\bK}{\mathbb{K}}
\Mdef{\bL}{\mathbb{L}}
\Mdef{\bM}{\mathbb{M}}

\Mdef{\bO}{\mathbb{O}}
\Mdef{\bP}{\mathbb{P}}
\Mdef{\bQ}{\mathbb{Q}}
\Mdef{\bR}{\mathbb{R}}
\Mdef{\bS}{\mathbb{S}}
\Mdef{\bT}{\mathbb{T}}
\Mdef{\bU}{\mathbb{U}}
\Mdef{\bV}{\mathbb{V}}
\Mdef{\bW}{\mathbb{W}}
\Mdef{\bX}{\mathbb{X}}
\Mdef{\bY}{\mathbb{Y}}
\Mdef{\bZ}{\mathbb{Z}}

\Mdef{\bbS}{\mathbb{S}}

\Mdef{\scrA}{\mathscr{A}}
\Mdef{\scrB}{\mathscr{B}}
\Mdef{\scrC}{\mathscr{C}}
\Mdef{\scrD}{\mathscr{D}}
\Mdef{\scrE}{\mathscr{E}}
\Mdef{\scrF}{\mathscr{F}}
\Mdef{\scrG}{\mathscr{G}}
\Mdef{\scrH}{\mathscr{H}}
\Mdef{\scrI}{\mathscr{I}}
\Mdef{\scrJ}{\mathscr{J}}
\Mdef{\scrK}{\mathscr{K}}
\Mdef{\scrL}{\mathscr{L}}
\Mdef{\scrM}{\mathscr{M}}
\Mdef{\scrN}{\mathscr{N}}
\Mdef{\scrO}{\mathscr{O}}
\Mdef{\scrP}{\mathscr{P}}
\Mdef{\scrQ}{\mathscr{Q}}
\Mdef{\scrR}{\mathscr{R}}
\Mdef{\scrS}{\mathscr{S}}
\Mdef{\scrT}{\mathscr{T}}
\Mdef{\scrU}{\mathscr{U}}
\Mdef{\scrV}{\mathscr{V}}
\Mdef{\scrW}{\mathscr{W}}
\Mdef{\scrX}{\mathscr{X}}
\Mdef{\scrY}{\mathscr{Y}}
\Mdef{\scrZ}{\mathscr{Z}}

\Mdef{\mcA}{\mathcal{A}}
\Mdef{\mcB}{\mathcal{B}}
\Mdef{\mcC}{\mathcal{C}}
\Mdef{\mcD}{\mathcal{D}}
\Mdef{\mcE}{\mathcal{E}}
\Mdef{\mcF}{\mathcal{F}}
\Mdef{\mcG}{\mathcal{G}}
\Mdef{\mcH}{\mathcal{H}}
\Mdef{\mcI}{\mathcal{I}}
\Mdef{\mcJ}{\mathcal{J}}
\Mdef{\mcK}{\mathcal{K}}
\Mdef{\mcL}{\mathcal{L}}
\Mdef{\mcM}{\mathcal{M}}
\Mdef{\mcN}{\mathcal{N}}
\Mdef{\mcO}{\mathcal{O}}
\Mdef{\mcP}{\mathcal{P}}
\Mdef{\mcQ}{\mathcal{Q}}
\Mdef{\mcR}{\mathcal{R}}
\Mdef{\mcS}{\mathcal{S}}
\Mdef{\mcT}{\mathcal{T}}
\Mdef{\mcU}{\mathcal{U}}
\Mdef{\mcV}{\mathcal{V}}
\Mdef{\mcW}{\mathcal{W}}
\Mdef{\mcX}{\mathcal{X}}
\Mdef{\mcY}{\mathcal{Y}}
\Mdef{\mcZ}{\mathcal{Z}}

\Mdef{\tA}{\tilde{A}}
\Mdef{\tB}{\tilde{B}}
\Mdef{\tC}{\tilde{C}}
\Mdef{\tE}{\tilde{E}}
\Mdef{\tH}{\tilde{H}}
\Mdef{\tK}{\tilde{K}}
\Mdef{\tL}{\tilde{L}}
\Mdef{\tM}{\tilde{M}}
\Mdef{\tN}{\tilde{N}}
\Mdef{\tP}{\tilde{P}}

\Mdef{\Ab}{\overline{A}}
\Mdef{\Bb}{\overline{B}}
\Mdef{\Cb}{\overline{C}}
\Mdef{\Db}{\overline{D}}
\Mdef{\Eb}{\overline{E}}
\Mdef{\Fb}{\overline{F}}
\Mdef{\Gb}{\overline{G}}
\Mdef{\Hb}{\overline{H}}
\Mdef{\Ib}{\overline{I}}
\Mdef{\Jb}{\overline{J}}
\Mdef{\Kb}{\overline{K}}
\Mdef{\Lb}{\overline{L}}
\Mdef{\Mb}{\overline{M}}
\Mdef{\Nb}{\overline{N}}
\Mdef{\Ob}{\overline{O}}
\Mdef{\Pb}{\overline{P}}
\Mdef{\Qb}{\overline{Q}}
\Mdef{\Rb}{\overline{R}}
\Mdef{\Sb}{\overline{S}}
\Mdef{\Tb}{\overline{T}}
\Mdef{\Ub}{\overline{U}}
\Mdef{\Vb}{\overline{V}}
\Mdef{\Wb}{\overline{W}}
\Mdef{\Xb}{\overline{X}}
\Mdef{\Yb}{\overline{Y}}
\Mdef{\Zb}{\overline{Z}}

\newcommand{\co}{\colon}

\def\endash{\mathchar"707B}

\newcommand{\leftmod}{\endash \textnormal{mod}}

\newcommand{\lra}{\longrightarrow}

\Mdef{\id}{\mathrm{Id}}

\newcommand{\adjunction}[4]{
\xymatrix{
#1:#2 \ar@<0.7ex>[r] &
\ar@<0.7ex>[l] #3:#4
}}

\Mdef{\bhom}{\mathbf{\hat{H}om}}

\Mdef{\Mod}{\mathrm{mod}}

\newcommand{\colim}{\mathop{  \mathrm{colim }} }

\DeclareMathOperator{\mackeyfunctor}{Mackey}
\DeclareMathOperator{\sheaffunctor}{Sheaf}

\newcommand{\weylsheaf}[1]{\textnormal{Weyl} \endash #1 \endash \textnormal{sheaf}_{\bQ}(\sub #1)}

\newcommand{\mackey}[1]{\textnormal{Mackey}_{\bQ}(#1)}

\newcommand{\sub}{\mathcal{S}}

\newcommand{\opensub}{\underset{\textnormal{open}}{\leqslant}}

\newcommand{\closedsub}{\underset{\textnormal{closed}}{\leqslant}}
\newcommand{\closedsuper}{\underset{\textnormal{closed}}{\geqslant}}

\newcommand{\opennormalsub}{\underset{\textnormal{open}}{\trianglelefteqslant}}

\DeclareMathOperator{\stab}{\textnormal{stab}}

\DeclareMathOperator{\core}{\textnormal{Core}}

\newcommand{\burnsidering}{\textbf{\textsf{\upshape{A}}}_\bQ} 
\DeclareMathOperator{\sky}{Sky}

\newcommand{\idemset}[2]{e^{#1}_{#2}}
\newcommand{\idem}[3]{e^{#1}_{#2,#3}}

\title[Weyl sheaves and Mackey functors]{The equivalence between rational \texorpdfstring{$G$}{G}-sheaves 
and rational \texorpdfstring{$G$}{G}-Mackey functors for profinite \texorpdfstring{$G$}{G}}

\author[Barnes]{David Barnes}
\address[Barnes]{Mathematical Sciences Research Centre, Queen's University Belfast}
\author[Sugrue]{Danny Sugrue} 
\address[Sugrue]{Mathematical Sciences Research Centre, Queen's University Belfast}

\thanks{
The second author gratefully acknowledges support from the
Engineering and Physical Sciences Research Council under Grant 1631308.
The authors wish to thank the anonymous referee for their many useful suggestions that improved the paper.
}

\begin{document}
\begin{abstract}
For $G$ a profinite group, we construct an equivalence between rational $G$-Mackey functors
and a certain full subcategory of $G$-sheaves over the space of closed subgroups of $G$ 
called Weyl-$G$-sheaves. This subcategory consists of those sheaves 
whose stalk over a subgroup $K$ is $K$-fixed. 

This extends the classification of rational $G$-Mackey functors for finite $G$
of Th\'{e}venaz and Webb, and Greenlees and May to a new class of examples. 
Moreover, this equivalence is instrumental in the classification of rational $G$-spectra
for profinite $G$, as given in the second author's thesis.
\end{abstract}

\maketitle

\tableofcontents

\newpage

\section{Introduction}

The classification of rational Mackey functors for finite groups 
is well-known and highly useful. 
An algebraic version of this result is given by 
Th\'{e}venaz and Webb \cite{tw95}. 
A version more suited to algebraic topology can be found in
work of Greenlees and May \cite[Appendix A]{gremay95}.
In this case the classification is given as an equivalence of categories
\begin{align}
\mackey{G} \cong \prod_{(H) \leqslant G} \bQ[W_G H] \leftmod \label{eqn:mackeyclass}
\end{align}
where the product runs over $G$-conjugacy classes of subgroups of $G$.
Greenlees and May proved this result in order to classify rational 
$G$-spectra (for finite $G$) in terms of an algebraic model. 
A gentle introduction to classifying rational $G$-spectra
(and similar results) can be found in work of the first author and K\k{e}dziorek 
\cite{BKclassify}. 

Profinite groups (compact Hausdorff totally-disconnected topological groups)
are an important class of topological groups that contain all finite groups. 
Examples of where profinite groups occur 
include the automorphism groups of Galois extensions of fields, 
the \'{e}tale fundamental groups of algebraic geometry
and the Morava stabilizer group $\mathbb{S}_n$ from chromatic homotopy theory.
The standard example is the $p$-adic integers $\bZ_p^\wedge$ for $p$ a prime. 

The main result of this paper is an extension of the classification 
of rational Mackey functors to profinite groups. 
\begin{xxthm}\label{mainthm}
For $G$ a profinite group there is an exact equivalence 
\[
\adjunction{\mackeyfunctor}{\weylsheaf{G}}{\mackey{G}}{\sheaffunctor.}
\]
\end{xxthm}
Here $\sub G$ denotes the space of closed subgroups of $G$, 
see Definition \ref{def:SG} for the topology.
A Weyl-$G$-sheaf over $\sub G$ is a $G$-equivariant sheaf over $\sub G$
such that the stalk at a closed subgroup $K \leqslant G$ is $K$-fixed.
Hence, the stalk at $K$ has an action of the Weyl group $W_G K$. 
This result appears as Theorem \ref{thm:equivalencemain} in the main body.

The reason why Mackey functors for finite groups have a simpler description is that 
$\sub G$ is a discrete space when $G$ is finite. In which case, a 
Weyl-$G$-sheaf over $\sub G$ is a product of $\bQ$-modules with 
actions of the Weyl groups. Choosing representatives for conjugacy classes 
gives the description (\ref{eqn:mackeyclass}) above. 
Furthermore, the rational Burnside ring for a profinite group 
is isomorphic to the ring of continuous functions from $\sub G /G$  to $\bQ$.
Hence, when $G$ is finite, the rational Burnside ring is simply a product of copies of $\bQ$, 
indexed over the set of conjugacy classes of subgroups of $G$.

Our next result, see Lemma \ref{lem:fixinflate}, extends a vital 
identity on rational $G$-Mackey functors from the finite case to the profinite case.
In the finite setting, the restriction map induces an isomorphism 
\begin{align}
e_K^H M(H)
\cong
e_K^K M(K)^{N_H(K)}\label{eq:diagdecomp} 
\end{align}
where $e^H_K$ is the idempotent in the rational Burnside ring
of $H$ corresponding to the conjugacy class of $K \leqslant H$.
To extend this to the profinite case, we use the idempotents $\idem{H}{N}{K}$ and $\idem{K}{N}{K}$
of Definition \ref{def:idempotentnames}, which are analogous to  $e^H_K$ and $e^K_K$.

\begin{xxlem}
Let $K \leqslant H$ be open subgroups of $G$ and $N$ an open normal subgroup of $K$.
Restriction from $H$ to $K$ induces an isomorphism
\[
\idem{H}{N}{K} M(H)
\cong
\idem{K}{N}{K} M(K)^{N_H(K)}.
\]
\end{xxlem}

The following is Proposition \ref{prop:Weylequi} and relates the equivariance of the stalks to that 
of representing sections. This proposition and the preceding lemma are key to proving 
the equivalence of categories of Theorem \ref{mainthm}. 

\begin{xxprop}
Let $E$ be a Weyl-$G$-sheaf over $\sub G$ and $K$ a closed subgroup of $G$.
A germ over $K$ can be represented by an $NK$-equivariant section over a neighbourhood of $K$,
for some open normal subgroup $N$ of $G$.
\end{xxprop}

This work is part of the second author's PhD thesis, \cite{sugruethesis},
completed under the supervision of the first author. 
That thesis applies the main result of this paper to obtain a classification of
rational $G$-spectra for profinite $G$ in terms of 
chain complexes of Weyl-$G$-sheaves over $\sub G$,
see \cite[Chapters 3 and 6]{sugruethesis}. 
This extends the work
of the first author on the case where $G$ is the $p$-adic integers, \cite{barnespadic}.
A further application of the classification of rational $G$-Mackey functors in terms of sheaves
is a calculation of the injective dimension of these categories in terms of the 
Cantor-Bendixson rank of the space $\sub G$,
see \cite[Chapter 8]{sugruethesis}.

\subsection{Organisation}

We define the categories involved in Section \ref{sec:basics} 
and construct the two functors of the classification in
Section \ref{sec:functors}. 
We prove directly that the functors give an equivalence in Section \ref{sec:equivalence}.
We end the paper with some examples and the relation between $G$-sheaves over $\sub G$
and Weyl-$G$-sheaves over $\sub G$, see Section \ref{sec:consequences}.

\section{Mackey functors and sheaves}\label{sec:basics}

We introduce the definitions and categories that are used in the paper. 

\subsection{Basic facts on profinite groups}

We give a few reminders of useful facts on profinite groups.
More details can be found in Wilson \cite{wilson98} or Ribes and Zalesskii \cite{rz00}.

A \textbf{profinite group} is a compact, Hausdorff, totally disconnected
topological group.
A profinite group $G$ is homeomorphic to the inverse limit of its finite quotients:
\[
G \cong \underset{N \opennormalsub G}{\lim}  G/N \subseteq \underset{N \opennormalsub G}{\prod} G/N .
\]
The limit has the canonical topology which can either be described as the subspace topology on the product
or as the topology generated by the pre-images of the open sets in
$G/N$ under the projection map $G \to G/N$, as $N$ runs over the open normal subgroups of $G$.

Closed subgroups and quotients by closed subgroups of profinite groups are also profinite.
A subgroup of a profinite group is open if and only if it is finite index and closed.
The trivial subgroup $\{ e\}$ is open if and only
if the group is finite.
The intersection of all open normal subgroups is $\{ e \}$.
Any open subgroup $H$ contains an open normal subgroup, the \textbf{core} of $H$ in $G$,
defined as the following intersection. Note that as $H$ has finite index in $G$, 
this intersection has only finitely many distinct terms. 
\[
\core_G (H) = \bigcap_{g \in G} g H g^{-1}.
\]

We can also define a profinite topological space
to be a Hausdorff, compact and totally disconnected topological space.
As with profinite groups, such a space is homeomorphic to
the inverse limit of a filtered diagram of finite spaces.
Moreover, a profinite topological space has a closed-open basis.

\subsection{Mackey functors for profinite groups}

There are many equivalent definitions of Mackey functors for a finite group $G$.
Our preferred version is a collection of abelian groups indexed over the subgroups
of $G$ with induction (transfer), restriction and conjugation maps satisfying a list of axioms.
In turn, there are several (inequivalent) extensions to profinite groups.
Since we have in mind applications to algebraic topology and the work of Fausk \cite{fausk},
we choose a generalisation that emphasises the role of the open subgroups.

The following definition can be found in Bley and Boltje
\cite[Definition 2.1 and Examples 2.2]{BB04} as the ``finite natural Mackey system''
and in Thiel \cite[Definition 2.2.12]{Thiel}.
In the case of a finite group, it restricts to the usual definition.

\begin{definition}\label{defn:mackey}
Let $G$ be a profinite group, then a rational \textbf{$G$-Mackey functor}\index{Mackey functor} $M$ is:
\begin{itemize}
\item a collection of $\mathbb{Q}$-modules $M(H)$ for each open subgroup $H \leqslant G$,
\item for open subgroups $K,H \leqslant G$ with $K\leqslant H$ and any $g \in G$
we have a \textbf{restriction} map, an \textbf{induction} map and a \textbf{conjugation} map
\[
R^H_K \co M(H)\rightarrow M(K), \quad I^H_K \co M(K)\rightarrow M(H) \quad \text{ and } \quad  C_g \co M(H)\rightarrow M(gHg^{-1}).
\]
\end{itemize}
These maps satisfy the following conditions.

\begin{enumerate}
\item For all open subgroups $H$ of $G$ and all  $h\in H$ the structure maps
are \textbf{unital}
\[
R^H_H=\id_{M(H)}=I^H_H \quad \text{ and } \quad  C_h=\id_{M(H)}.
\]

\item For $L\leqslant K\leqslant H$ open subgroups of $G$ and $g,h\in G$, there are composition rules
\[
I^H_L=I^H_K\circ I^K_L, \qquad R^H_L=R^K_L\circ R^H_K, \quad \text{ and } \quad  C_{gh}=C_g\circ C_h.
\]
The first two are \textbf{transitivity} of induction and restriction. The last is \textbf{associativity} of conjugation.

\item For $g\in G$ and $K\leqslant H$ open subgroups of $G$, there are composition rules
\[
R^{gHg^{-1}}_{gKg^{-1}}\circ C_g=C_g\circ R^H_K  \quad \text{ and } \quad  I^{gHg^{-1}}_{gKg^{-1}}\circ C_g=C_g\circ I^H_K.
\]
This is the \textbf{equivariance} of restriction and induction.

\item For open subgroups $K,L\leqslant H$ of $G$
\[
R^H_K\circ I^H_L=\sum_{x\in\left[ K \backslash H \slash L \right]} I^K_{K\cap xLx^{-1}}\circ C_x\circ R^L_{L\cap x^{-1}Kx}.
\]
This condition is known as the \textbf{Mackey axiom}.
\end{enumerate}

A map of Mackey functors $f \co M \to N$
is a collection of $\mathbb{Q}$-module maps  
\[
f(H) \co M(H) \lra N(H)
\]
for each open subgroup $H \leqslant G$, which commute with the induction, restriction
and conjugation maps. 
We use $\mackey{G}$ to denote this category.
\end{definition}
See Th\'{e}venaz and Webb \cite{tw95} for an introduction into Mackey functors for finite groups.
The case of rational Mackey functors is also discussed in work of the first author with
K\k{e}dziorek \cite{BKclassify}.
Alternative definitions of Mackey functors in a general setting are discussed in
Lindner \cite[Definition 1]{lindner}.

We have some standard examples similar to those for finite groups.
\begin{example}
Consider the profinite group given by the $p$-adic integers, $\mathbb{Z}_p$. The open subgroups in this case are of the form $p^k\mathbb{Z}_p$ for $k\geqslant 0$.
We define a constant Mackey functor by sending each open subgroup to $\mathbb{Q}$
with all restriction and conjugation maps taken to be the identity.
The subgroup $p^l\mathbb{Z}_p$ is contained inside $p^k\mathbb{Z}_p$ if and only if $l\geqslant k$.
The Mackey axiom forces the induction maps to be
\[
I^{p^k\mathbb{Z}_p}_{p^l\mathbb{Z}_p}=\,\text{Multiplication by } p^{l-k}.
\]

In general, there is a constant Mackey functor for any profinite group $G$,
which takes value $\bQ$ at every open subgroup $H$ of $G$.
The conjugation and restriction maps are identity maps and
induction from $K$ to $H$ is multiplication by the index of $K$ in $H$.
\end{example}

\begin{example}
Let $R(G)$ denote the ring of finite dimensional
complex representations of the profinite group $G$.
We define a rational Mackey functor $M_R$ by $M_R(H) = R(H) \otimes \bQ$,
with induction and restriction induced by induction and restriction of representations.
Conjugation is induced from precomposition with the conjugation homomorphism of groups.
\end{example}

In the case of finite groups, another common example is
the fixed point Mackey functor. A profinite analogue exists,
but one needs to be more careful over the action of $G$
on a $\bQ$-module.

\begin{definition}
A rational \textbf{discrete} $G$-module $M$ is a $\bQ$-module
such that
\[
\underset{H \opensub G}{\colim} M^H \cong M.
\]
\end{definition}
The definition is equivalent to the statement
that any $m \in M$ is stabilised
by an open subgroup of $G$.
If we give a rational $G$-module $M$ the discrete topology,
then the action of $G$
on $M$ is continuous if and only if $M$ is discrete in the preceding sense.

\begin{example}\label{ex:fixpointmackey}
Given $M$ a discrete $G$-module, there is a Mackey functor
which at an open subgroup $H \leqslant G$ takes value $M^H$.
The restriction maps are inclusion of fixed points
and the conjugation maps are given by left multiplication.
The induction map from $K$ to $H$ is given by summing over the action of
a left transversal of $K$ in $H$ (which is finite as the subgroups are open). 
\end{example}

\begin{example}
The rational Burnside ring $\burnsidering(G)$ of a profinite group $G$
defines a rational Mackey functor.
For $H$ an open subgroup of $G$, the ring $\burnsidering(H)$
is the rational Grothendieck ring of finite $H$-sets
(that is, finite discrete topological spaces with
a continuous action of $H$).
The multiplication is given by the product of $H$-sets.

Restriction and induction between
$\burnsidering(H)$ and $\burnsidering(K)$
are given by the usual
restriction and induction of sets with group actions.
Moreover, the restriction maps are maps of rings.
Conjugation is induced from precomposition with the conjugation homomorphism of groups.
\end{example}

Given a $G$-Mackey functor $M$, we can define an action of the
Burnside ring $\burnsidering(H)$ on the
abelian group $M(H)$ by
\[
[H/K] \coloneqq  I_{K}^{H} R_{K}^{H} \co M(H) \lra M(H)
\]
and extending linearly from the additive basis for $\burnsidering(H)$ given by
$H/K$ for open subgroups $K$ of $H$.
The Mackey axiom implies that this action is compatible with the
multiplication of $\burnsidering(H)$, so that $M(H)$ is a module over
$\burnsidering(H)$. Moreover, the following square commutes.
\[
\xymatrix{
\burnsidering(H) \otimes M(H)
\ar[r] \ar[d]_{R_K^H \otimes R_K^H} &
M(H) \ar[d]_{R_K^H} \\
\burnsidering(K) \otimes M(K)
\ar[r] &
M(K)
}
\]
We also have the relation, known as \textbf{Frobenius reciprocity}, 
between induction and the action of Burnside rings. 
Let $\alpha \in \burnsidering(H)$, $\beta \in \burnsidering(K)$,
$m \in M(K)$ and $n \in M(H)$, then 
\[
\alpha \cdot I_K^H (m) = I_K^H ( R_K^H (\alpha) \cdot m)
\quad \quad
I_K^H(\beta) \cdot n = I_K^H ( \beta \cdot R_K^H(n)).
\]
See Yoshida \cite[Definition 2.3 and Example 2.11]{yosh80}.
We summarise this discussion in the following proposition.

\begin{proposition}\label{prop:mackeyact}
If $M$ is a Mackey functor for a profinite group $G$ and $H$ is an open subgroup of $G$, then each $M(H)$ has
an $\burnsidering(H)$-module structure. Moreover, the restriction map $M(H) \to M(K)$ is a map of modules over
the restriction map $\burnsidering(H) \to \burnsidering(K)$, which is compatible with the conjugation maps.
The Burnside ring actions satisfy Frobenius reciprocity with respect to the induction maps. 
\end{proposition}

\subsection{Burnside ring idempotents}

As with finite groups, there is a description of
the rational Burnside ring in terms of continuous maps
out of the space of subgroups of $G$.
We will use this structure repeatedly, so we examine the
result in detail and use it to describe idempotents of the Burnside ring.

\begin{definition}\label{def:SG}
For $G$ a profinite group, the space of \emph{closed} subgroups of $G$, $\sub G$,
is the topological space
\[
\sub G \coloneqq \underset{N \opennormalsub G}{\lim}  \sub (G/N)
\]
where $\sub (G/N)$ is a finite space with the discrete topology. 
For $N \leqslant N'$, the corresponding map of the limit diagram is
\[
p_{N,N'} \co \sub (G/N) \lra \sub (G/N') \qquad H/N \longmapsto HN'/N'.
\]
The limit is topologised as either the subspace of the product
or as the coarsest topology so that the projection maps
\[
p_N \co \sub G \lra \sub (G/N) \qquad H \longmapsto HN/N
\]
are continuous. The group $G$ acts continuously on $\sub G$ by conjugation,
$g \in G$ sends a closed subgroup $K$ to $gK g^{-1} = \prescript{g}{}{K}$.
\end{definition}

Following Gartside and Smith \cite[Lemma 2.14]{GartSmith10}, we can give a 
compact-open basis for this topology. 

\begin{lemma}
An explicit basis for $\sub G$ is the collection of closed and open sets
\[
O_G(N, J) = \{ K \in \sub G \mid NK =J \}
\]
such that $N$ is an open normal subgroup of $G$, and $J$ is an open subgroup of $G$
(that contains $N$).
\end{lemma}
\begin{proof}
The set $O_G(N, J)$ is the pre-image of the set $\{ J/N \} \in \sub (G/N)$ under $p_N$,
hence it is open and closed.
To see this is a basis, let $A \in O_G(N, J) \cap O_G(N', J')$ and set 
$N''=N \cap N'$, then 
\[
A \in O_G(N'', N''A) \subseteq O_G(N, J) \cap O_G(N', J'). \qedhere
\]
\end{proof}

These sets are compatible with the $G$-action:
\[
gO_G(N,J) = \{\prescript{g}{}{K} \mid NK=J \} = O_G(N,\prescript{g}{}{J}).
\]
Since any open subgroup of $G$ has form $N K$ for $N$ some open normal subgroup of $G$
and $K$ a closed subgroup of $G$, we can rewrite the basis to consist of the sets
\[
\Big\{ O_G(N, NK) \Bigm| N \opennormalsub G,  \ \ K \closedsub G \Big\}.
\]
If we fix $K$ and allow $N$ to vary, these
sets give a neighbourhood basis for $K \in \sub G$.
We further define 
\[
\overline{O}_G(N,J) = O_G(N,J)/G \subseteq \sub G/G,
\]
the quotient under the conjugation action.
This set is open in $\sub G/G$ as the quotient map is open. It is closed as 
$O_G(N,K)$ is compact. 
We will also make use of the equality 
\[
O_G(N,NK) = O_{NK}(N,NK).
\]

\begin{theorem}\label{thm:burnchar}
For $G$ a profinite group, there is an isomorphism of commutative rings
\[
\burnsidering(G)\cong C(\sub G/G,\mathbb{Q}).
\]
\end{theorem}
\begin{proof}
This is Dress \cite[Theorem B.2.3(a)]{Dress}. As with finite groups, the isomorphism
maps $[G/H]$ to the function which sends a conjugacy class of closed subgroup
$K$ to $| (G/H)^K |$.
\end{proof}

\begin{lemma}
For $G$ a profinite group, there is a natural isomorphism of commutative rings
\[
\underset{N \opennormalsub G}{\colim} \burnsidering(G/N)
\cong
\burnsidering(G).
\]
\end{lemma}
\begin{proof}
Since any $G/N$-set is a $G$-set by precomposing with the quotient $G \to G/N$,
we have the given map of rings, which is injective.
Surjectivity follows as a continuous action of $G$ on a finite set factors
through some finite quotient of $G$.
\end{proof}

\begin{remark}
Similar to Pontryagin duality, the previous lemma and
theorem imply that the canonical map
\[
\underset{N \opennormalsub G}{\colim} C(\sub (G/N)/(G/N),\mathbb{Q})
\cong
C \Big( \underset{N \opennormalsub G}{\lim}  \sub (G/N)/(G/N),\mathbb{Q} \Big)
=
C(\sub G/G,\mathbb{Q})
\]
induced by the surjections $\sub G/G \to \sub (G/N)/(G/N)$ is an isomorphism.

This type of isomorphism holds for any profinite space $X$.
The difficult step is proving surjectivity. This comes from the
fact that any map $X \to \bQ$ must have finite image and hence factor through some
finite quotient of $X$. See Ribes and Zalesskii \cite[Proposition 1.1.16 (a) and Lemma 5.6.3]{rz00}.
\end{remark}

An idempotent of $\burnsidering(H)=C(\sub H/H,\mathbb{Q})$ corresponds
to an open and closed subset of $\sub H/H$.
The projection map $p_H \co \sub H \to \sub H/H$ is
open and both domain and codomain are compact, 
hence an open and closed subset of $\sub H$
defines an idempotent of $\burnsidering(H)$. 
Indeed, the set of idempotents of $\burnsidering(H)$
is in bijection with the set of open, closed and $H$-invariant subsets of $\sub H$.

\begin{definition}\label{def:idempotentnames}
For $U$ an open and closed subset of $\sub H/H$, there is an idempotent
$\idemset{H}{U}$, the characteristic function of $U$.

An open and closed subset $V$ of $\sub H$
(that is not necessarily $H$-invariant)  defines an idempotent
$\idemset{H}{V} \coloneqq  \idemset{H}{p_H(V)}$ of $\burnsidering(H)$, the characteristic 
function of $V/H$. In turn, we define
\[
\idem{H}{N}{J} = \idemset{H}{O_H(N,J)} = \idemset{H}{\overline{O}_H(N,J)}
\]
the characteristic function for $\overline{O}_H(N,J)=O_H(N,J)/H$, for $J$ an open subgroup of $H$.
\end{definition}

We relate our closed-open basis for $\sub G$
to idempotents of $C(\sub G/G,\mathbb{Q})$
and the additive basis of that ring.
For the formula we will use
the Moebius function $\mu$ for pairs of subgroups
$D\leqslant K$
\begin{align*}
\mu(D,K)=\sum_{i \geqslant 0} (-1)^i c_i
\end{align*}
where $c_i$ is the number of strictly increasing chains from $D$ to $K$ of length $i$.
Note that a chain of the form $D < K$ has length 1 and, by convention, $\mu(K,K)=1$.

\begin{proposition}\label{prop:idempotents}
For $G$ a profinite group, with open normal subgroup $N$ and
closed subgroup $K$
\begin{align*}
\idem{G}{N}{NK}=\sum_{N \leqslant NA\leqslant NK}\frac{|NA|}{|N_G(NK)|}\mu(NA,NK) [G/NA].
\end{align*}
\end{proposition}
\begin{proof}
The sum is finite as we range over open subgroups of $G$ that contain $N$.

The method is to reduce the calculation to the case of finite groups,
and use the formula for idempotents of the rational Burnside ring of
the finite group $G/N$
from Gluck \cite[Section~2]{gluck81}.
The reduction comes from the following commuting diagrams, 
where $p_N \co \sub G \to \sub (G/N)$ sends $H$ to $HN/N$.
The second diagram uses the homeomorphism of $G$-spaces:
\begin{align*}
G/NA\cong (G/N)/(NA/N).
\end{align*}
\[
\xymatrix{
\sub G
\ar[rr]^{\idem{G}{N}{NK}}
\ar[dr]_{p_N}
&& \bQ \\
& \sub (G/N) \ar[ur]_{e_{NK/N}}
}
\qquad
\xymatrix{
\sub G
\ar[rr]^{K \mapsto |(G/NA)^K|}
\ar[dr]_{p_N}
&& \bQ \\
& \sub (G/N) \ar[ur]_{K \mapsto |(G/N)/(NA/N)^K|}
}
\]
If we replace $p_N$ by $\overline{p}_N \co \sub G/G \to \sub (G/N)/(G/N)$
the two diagrams above remain commutative. 

The first diagram and the formula for idempotents gives
\begin{align*}
\idem{G}{N}{NK}=
\underset{N \leqslant NA\leqslant NK}{\sum}\frac{|NA|}{|N_G(NK)|}\mu(NA,NK)
\left[\left(G/N\right)/\left(NA/N\right)\right]\circ p_N.
\end{align*}
The second diagram completes the proof.
\end{proof}

The coefficients are constant under certain conjugations. More precisely, 
for $K\leqslant J\leqslant G$ open subgroups of $G$ and $a \in N_G(J)$,
if we define
\[
\alpha_{K,J}=\frac{|K|}{|N_G(J)|}\mu(K,J),
\quad
\text{ then }
\quad
\alpha_{K,J}=\alpha_{aKa^{-1},J}.
\]

A simple observation shows that
for a rational $G$-Mackey functor $M$, each
component $M(H)$ is a sheaf over the profinite space $\sub H / H$.
This holds as 
$M(H)$ is a module over $\burnsidering(H) = C(\sub H/H, \bQ)$.

\begin{proposition}\label{prop:moduletosheaf}
Suppose $X$ is a profinite space and $M$ is a $C(X,\mathbb{Q})$-module, then $M$ determines a sheaf of $\bQ$-modules over $X$.
\end{proposition}
\begin{proof}
Let $\mathfrak{B}$ be an open-closed basis for $X$.
Given a module $M$, we define a sheaf $F$ by giving its values on $\mathfrak{B}$.
Let $U \in \mathfrak{B}$ and define the characteristic function of $U$
to be the map which sends elements of $U$ to 1 and the rest of $X$ to 0.
This map, which is denoted $e_U \in C(X,\mathbb{Q})$, is continuous as $U$ is both open and closed.
We define $F(U) = e_U M$. For $V \subseteq U$ in $\mathfrak{B}$,
we define the restriction map to be the projection
(where $V^c$ is the complement of $V$ in $X$)
\[
e_U M \cong e_V M \oplus e_{V^c \cap U} M \lra e_V M
\]
which is equivalent to multiplying by $e_V$.

The sheaf condition follows as we have a closed-open basis of a compact space and hence
can write any open cover of a basis element as a finite partition of basis elements.
\end{proof}

\begin{corollary}\label{cor:mackeytermsheaf}
If $M$ is a Mackey functor over a profinite group $G$, then for each open subgroup $H$ of $G$,
the $\bQ$-module  $M(H)$ defines a sheaf over $\sub H/H$,
the space of $H$-conjugacy classes of closed subgroups of $H$.
\end{corollary}

Paraphrasing the corollary, we see that a $G$-Mackey 
functor is a collection of sheaves, one for each open subgroup of $G$.
The rough idea of the equivalence is to patch parts of these sheaves together
to construct a $G$-equivariant sheaf.
It will take some time to make this rigorous, but we can give a useful relation
between the various sheaves $M(H)$. This relation is an extension of a result for finite groups,
as referenced in the proof, see also Equation (\ref{eq:diagdecomp}) of the introduction.

\begin{lemma}\label{lem:fixinflate}
Let $K \leqslant H$ be open subgroups of $G$ and $N$ an open normal subgroup of $K$.
Then restriction from $H$ to $K$ gives an isomorphism
\[
\idem{H}{N}{K} M(H)
\cong
\idem{K}{N}{K} M(K)^{N_H(K)}.
\]
\end{lemma}
\begin{proof}
The method is to reduce the problem to the case of finite groups.
Similarly to Greenlees and May \cite[Definition 3]{GM92structure} or Th\'{e}venaz and Webb \cite{tw95},
we may define a $G/N$-Mackey functor $\overline{M}$ by
\[
\overline{M} (J/N) = M(J).
\]
Proposition \ref{prop:idempotents} implies that when $J \geqslant N$,
\[
\idem{J}{N}{K}  M(J)
\cong
e_{K/N}^{J/N} \overline{M} ( J/N ).
\]

Now that we have Mackey functors for finite groups, we may use
\cite[Example 5C(i) and Corollary 5.3]{greratmack} to see that
restriction induces an isomorphism
\[
e_{K/N}^{K/N} \overline{M} (K/N )^{N_{H/N}(K/N)}
\cong
e_{K/N}^{H/N} \overline{M} (H/N ).
\]
Alternative proofs of that result occur in
\cite[Lemma 6.1.9]{sugruethesis} and \cite{BKclassify}.
Combining these isomorphisms and the fact
$N_{H/N}(K/N)\cong N_H(K)/N$, we obtain
\begin{align*}
\idem{K}{N}{K}  M(K)^{N_H(K)}
&=
\idem{K}{N}{K}  M(K)^{N_H(K)/N}\\
&=
e_{K/N}^{K/N} \overline{M} ( K/N )^{N_{H/N}(K/N)} \\
&=
e_{K/N}^{H/N} \overline{M} ( H/N  ) \\
&=
\idem{H}{N}{K}  M(H). \qedhere
\end{align*}
\end{proof}

\subsection{Equivariant sheaves}

We begin with the definition of a $G$-equivariant sheaf over
a profinite $G$-space $X$ where $G$ is a profinite group.
The second author gives two equivalent definitions in
\cite{sugruethesis}, see also \cite{BSsheaves}.
We work with just one for brevity.

\begin{definition}\label{defn:eqsheaf}
A \textbf{$G$-equivariant sheaf} of $\mathbb{Q}$-modules over $X$ is a map
of topological spaces $p \co E \to X$ such that:
\begin{enumerate}
\item \label{item:sheafeq} $p$ is a $G$-equivariant map $p:E\rightarrow X$ of spaces with continuous $G$-actions,
\item \label{item:sheafab} $(E,p)$ is a sheaf space (\'etale space) of $\bQ$-modules,
\item \label{item:sheafcomb} each map $g \co p^{-1} (x) \rightarrow p^{-1} (g x)$ is a map of $\mathbb{Q}$-modules for every $x\in X,g\in G$.
\end{enumerate}
We will write this as either the pair $(E,p)$ or simply as $E$.
We call $E$ the \textbf{total space}, $X$ the \textbf{base space} and $p$ the structure map.

We call $p^{-1}(x)$ the \textbf{stalk} of $E$ at $x$ and denote it $E_x$.
An element of a stalk is called a \textbf{germ}.
\end{definition}
Note that points (\ref{item:sheafeq}) and (\ref{item:sheafab}) give a map of
sets for point (\ref{item:sheafcomb}),
but they do not imply that it is a map of $\mathbb{Q}$-modules.
Given a $G$-equivariant sheaf $(E,p)$ over $X$ and 
$x \in X$, the stalk $E_x$ (equipped with the discrete topology)
has a continuous action of $\stab_G(x)$.

As in the non-equivariant case, given an open subset $U \subseteq X$
the space of (continuous) sections
\[
E(U) = \Gamma(U,E) = \{ s \co U \lra E \mid p \circ s =\id_U \}
\]
has an addition operation (defined stalk-wise). The sections are not required to be
$G$-equivariant.
Allowing $U$ to vary defines a functor from the set of open subsets of
$X$ to $\mathbb{Q}$-modules.
Moreover, one can show that
\[
\underset{U \ni x}{\colim \,} \Gamma(U,E) = E_x=p^{-1} (x)
\]
as in the non-equivariant setting.

Given a section $s \co U \to E$ and $g \in G$, we can define
\[
g \ast s =  g \circ s \circ g^{-1} \co g U \lra E
\]
which sends $v=gu$ to $gs(u) = gs(g^{-1} v)$.
Hence, if $U$ is invariant under the action of a subgroup $H$
(that is, $h U = U$ for all $h \in H$), the space of sections
$\Gamma(U,E)$ has an $H$-action. The fixed points of this space are
those sections which commute with the action of $H$,
which we call \textbf{$H$-equivariant} sections.

\begin{definition}\label{defn:weyleqsheaf}
A \textbf{Weyl-$G$-sheaf} of $\mathbb{Q}$-modules over $\sub G$ is a $G$-sheaf of
$\mathbb{Q}$-modules over $\sub G$ such that action of $H$ on $E_H$ is trivial.
Hence $E_H$ has an action of the Weyl group $W_G H$ of $H$ in $G$.
We use $\weylsheaf{G}$ to denote this category.
\end{definition}

The idea of the classification result is to construct a Mackey functor $M$ from
a Weyl-$G$-sheaf $E$ by setting $M(H) = E(\sub H)^H$.
For this we need $\sub H$ to be an open subset of $\sub G$ when $H$ is open in $G$.

\begin{lemma}\label{lem:subHopenclosed}
For $H$ a closed subgroup of $G$,
$\sub H$ is closed in $\sub G$.
This subspace is also open when $H$ is open.
\end{lemma}
\begin{proof}
The open statement follows from noting that
\[
\sub H = \bigcup_{K \in \sub H} 
O_H(\core(H),  \core(H) K).
\]
For the closed statement we see that
$H$ is the limit of finite groups of the form $H/(H \cap N) =  HN/N$
for $N$ an open normal subgroup of $G$.
Since $\sub (HN/N)$ is a closed subset of $\sub (G/N)$,
the result follows by taking limits.
\end{proof}

We also need a result which can be described as saying that a section of
a $G$-equivariant sheaf over a profinite $G$-space is ``locally sub-equivariant''.
See \cite[Section 4.3]{sugruethesis} for related results. 

\begin{proposition}\label{prop:Weylequi}
If $E$ is a Weyl-$G$-sheaf over $\sub G$ and $K$ a closed subgroup of $G$, then any
$s_K\in E_K$ can be represented by an $NK$-equivariant section
\begin{align*}
s \co O_{NK}(N,NK)\lra E
\end{align*}
for $N$ some open normal subgroup of $G$.
\end{proposition}

\begin{proof}
Let $s \co O_{M K}(M ,M K)=U \to E$ be a section representing $s_K$. 
The set $s(U)$ is not necessarily closed under the action of $G$, 
but if $gu \in U$ and $gs(u) \in s(U)$ for some $g \in G$ and $u \in U$, then 
\[
ps(gu) = gu = gp(s(u)) = p(g s(u))
\]
as $p$ is $G$-equivariant. Since $p$ is injective when restricted to $s(U)$, 
we see that $s(gu)=g(s(u))$. The method of the proof is to restrict the domain and codomain
so that they are closed under the action of $NK$, for $N$ some open normal subgroup of $G$.

The set $s(U)$ is open (by definition of the topology on a sheaf space)
and is the image of compact set.  
Hence, there is an open normal subgroup $M'$ of $G$ such that $M' s(U)=s(U)$ 
by \cite[Lemma A.1]{BSsheaves}. 
Let 
\[
V= \bigcap_{k \in K} k s(U)
\]
which consists of only finally many distinct terms as 
$M' \cap K$ has finite index in $K$. Moreover, $s_K=s(K) \in V$
as $s(K)$ is $K$-fixed, so $V$ is a non-empty compact open subset of $E$
which is invariant under $M''K$ for some open normal subgroup $M''$ of $G$.

The set $p(V)$ is open and contains $K$, hence we can find a basic open set containing $K$
of the form
\[
W=O_{N K}(N ,N K) \subseteq O_{M K}(M ,M K) \cap p(V)
\]
for some $N \leqslant M \cap M''$.
The section $s_{\mid W}$ is $NK$-equivariant by our earlier argument as 
$gw \in W$ and $gs(w) \in V \subseteq s(U)$ for all $g \in NK$
and $w \in W$.
\end{proof}

\section{The functors}\label{sec:functors}
In this chapter we will construct a correspondence between rational $G$-Mackey functors
and Weyl-$G$-sheaves over $\sub G$.
We shall explicitly construct functors between the categories,
see Theorems \ref{thm:sheaftomack} and \ref{thm:macktosheaf}.
In Theorem \ref{thm:equivalencemain} we will see that these functors are equivalences
of categories.

\subsection{Weyl-\texorpdfstring{$G$}{G}-sheaves determine Mackey Functors}

We define a functor:
\[
\mackeyfunctor \co
\weylsheaf{G}
\rightarrow
\mackey{G}.
\]
We will not need the input sheaf to be a Weyl-$G$-sheaf, a detail we return to 
in Subsection \ref{subsec:weyltosheaf}

\begin{construction}\label{con:mackey_construct}
Let $(E,p)$ be a $G$-sheaf over $\sub G$, we  define a Mackey functor\index{Mackey functor} $\mackeyfunctor(E)$ as follows.

For $H\leqslant G$ an open subgroup define
\[
\mackeyfunctor(E)(H)=E( \sub H)^H
=
\{ s \co \sub H \lra E \mid p \circ s =\id_{\sub H} \}^H
\]
the set of $H$-equivariant sections on $\sub H$.
The conjugation maps $C_g$ are given by the $G$-action on the sheaf $E$.

For $K \leqslant H$ another open subgroup,
the restriction map
\[
R^H_K \co E(\sub H)^H\rightarrow E( \sub K)^K
\]
is given by restriction of a section to the subspace $\sub K$.

Let $T \subset H$ be a left transversal of $K$ in $H$.
For $\beta$ a section of $\sub K$, let $\overline{\beta}$
be the extension by zero to $\sub H$.
We define the induction map by
\[
I^H_K \co E( \sub K)^K\rightarrow E( \sub  H)^H, \quad
\beta \mapsto \underset{h \in T}{\sum} h \ast \overline{\beta}.
\]

We prove that the induction functor is well defined in Lemma \ref{lem:well_define}.
The structure maps compose and interact appropriately by Lemma \ref{lem:associativestructuremaps}.
Lemma \ref{lem:Mackey_axiom} proves that the Mackey axiom holds.
\end{construction}

\begin{lemma}\label{lem:well_define}
The induction map given in Construction \ref{con:mackey_construct} is well defined.
\end{lemma}
\begin{proof}
Let $hK=h'K$ be two representatives for the same coset.
Since the section $\beta$ is $K$-fixed,
\[
h \ast \overline{\beta} = h' \ast \overline{\beta}
\]
as $h$ and $h'$ differ by an element of $K$.
It follows that $I^H_K $ is independent of the choice of $T$.

The sum of $h \ast \overline{\beta}$ over $h \in T$ is $H$-equivariant
as $T$ is a transversal.
\end{proof}

\begin{lemma}\label{lem:associativestructuremaps}
The structure maps of $\mackeyfunctor(E)$ are unital, associative, transitive
and equivariant.
\end{lemma}
\begin{proof}
The part of most interest is that the induction maps are transitive
and equivariant.
The first follows from the fact that
given open subgroups $J \leqslant K \leqslant H$,
one can combine a transversal of $J$ in $K$ with a
transversal of $K$ in $H$ to get a transversal of $J$ in $H$.
This gives the transitivity.
For equivariance, the result follows as one can conjugate
a transversal to get a transversal of the conjugate.
\end{proof}

We note that $\sub H$ is invariant under the action of $N_G H \leqslant G$,
so $E( \sub H)$ has a continuous action of $N_G H$.

It remains to show the Mackey axiom, which is a direct calculation.
\begin{lemma}\label{lem:Mackey_axiom}
The construction $\mackeyfunctor(E)$ satisfies the Mackey axiom.
\end{lemma}
\begin{proof}
We start with $J,L\leqslant H$ where $H,J,L\leqslant G$ are open. We can decompose $H$
into double cosets.
\[
H
=
\underset{hL\in H/L}{\coprod} hL
=
\underset{x\in\left[J\backslash H\slash L\right]}{\coprod}JxL
=
\underset{x\in \left[J\backslash H\slash L\right]}{\coprod}\,\,\,\underset{j\in J}{\bigcup} jxL
=
\underset{\substack{x\in\left[J\backslash H\slash L\right]  \\ j_x \in  J/J\cap xLx^{-1}}}
{\coprod} j_x xL
\]

Given a transversal $T$ for $L$ in $H$ we have
\[
R^H_JI^H_L(\beta)
=
(\underset{h \in T}{\sum} h \ast \overline{\beta})\Big|_{\sub J}
=
\Big( \sum_{\substack{x\in\left[J\backslash H\slash L\right]  \\ j_x \in  J/J\cap xLx^{-1}}}
j_x x \ast \overline{\beta}
\Big) {\Big|}_{\sub J}
=
\sum_{\substack{x\in\left[J\backslash H\slash L\right]  \\ j_x \in  J/J\cap xLx^{-1}}}
\left(j_x x\overline{\beta|_{\sub \left(x^{-1}Jx\cap L\right)}}\right)
\]
where the extension by zero in the third term is with respect to $H$ and in the
in the fourth term it is with respect to $J$. If we start from the other direction we have:
\begin{align*}
\underset{x\in \left[J\backslash H\slash L\right]}{\sum}
I^J_{J\cap xLx^{-1}} \circ C_x \circ R^L_{x^{-1}Jx\cap L}(\beta)
&=
\sum_{\substack{x\in\left[J\backslash H\slash L\right]  \\ j_x \in  J/J\cap xLx^{-1}}}
j_x x \ast \left(\overline{\beta|_{S\left(L\cap x^{-1}Jx\right)}}\right)
\\
&=
\sum_{\substack{x\in\left[J\backslash H\slash L\right]  \\ j_x \in  J/J\cap xLx^{-1}}}
\left(j_x x\overline{\beta|_{S\left(L\cap x^{-1}Jx\right)}}\right)
\end{align*}
proving that the two sides coincide.
\end{proof}

Similar arguments to the above show that a map of sheaves
$E \to E'$ induces a map of $G$-Mackey functors
$\mackeyfunctor(E) \to \mackeyfunctor(E')$.
We summarise this section in the following theorem.

\begin{theorem}\label{thm:sheaftomack}
Let $G$ be a profinite group. If $E$ is a $G$-sheaf of $\bQ$-modules over $\sub G$,
then $\mackeyfunctor(E)$ from Construction \ref{con:mackey_construct}
is a Mackey functor and the assignment is functorial.
Hence, there is a functor
\begin{align*}
\mackeyfunctor \colon \weylsheaf{G} \rightarrow \mackey{G}.
\end{align*}
\end{theorem}

\subsection{Mackey Functors determine Weyl-\texorpdfstring{$G$}{G}-sheaves}

We construct a functor in the opposite direction, from rational $G$-Mackey functors to
Weyl-$G$-sheaves of $\bQ$-modules over $\sub G$.
We start by constructing the stalks of the Weyl-$G$-sheaf.

\begin{definition}\label{lem:mackeystalks}
For $M$ a rational $G$-Mackey functor
and $K$ a closed subgroup of $G$, we define a $\mathbb{Q}$-module
\[
\sheaffunctor(M)_K=\underset{J \geqslant K}{\colim} \, M(J)_{(K)} =
\underset{N \opensub G}{\colim} M(NK)_{(K)} =
\underset{N \opensub G}{\colim} \idem{NK}{N}{NK}  M(NK).
\]
The maps in the colimits are induced by the restriction maps of $M$ and applying idempotents.
The notation $M(J)_{(K)}$ refers to the stalk of
the $\sub J/J$-sheaf $M(J)$ at the $J$-conjugacy class of $K$, for $J$ an
open subgroup of $G$ containing $K$.
\end{definition}

The restriction maps are compatible
with taking stalks since
$\idem{H}{N}{NK} \in \burnsidering(H)$ restricts to an idempotent
$f \in \burnsidering(J)$ for $K \leqslant J \leqslant H$, where the support of $f$
is those subgroups of $J$ (up to $J$-conjugacy) which are $H$-conjugate
to an element of $O_H(N, NK)$. This support contains
\[
K \in O_J(N, NK) = O_H(N, NK) \cap \sub J.
\]

We begin our construction of a rational $G$-equivariant
sheaf $(E,p)$ over $\sub G$ from a Mackey functor $M$.
We first define $E$ and $p$.

\begin{construction}\label{con:sheaf_mack}
If $M$ is a rational $G$-Mackey functor,
we define the underlying set of the sheaf space by setting
\[
\sheaffunctor(M)=\underset{K\in \sub G}{\coprod} \sheaffunctor(M)_K.
\]
The projection map $p \co \sheaffunctor(M) \to \sub G$ sends all of $\sheaffunctor(M)_K$ to $K$.
The conjugation maps of $M$ induces maps
\[
\sheaffunctor(M)_K \lra \sheaffunctor(M)_{g K g^{-1}}
\]
for each $g \in G$. Thus $\sheaffunctor(M)$ has a $G$-action and $p$ is $G$-equivariant.

We construct a set of sections that will give a basis for a topology
on $\sheaffunctor(M)$. 
For $H$ an open subgroup of $G$, define a map 
\begin{align*}
\theta_H \colon M(H) 
\lra 
\Big\lbrace s \colon \sub  H\lra \coprod_{K \in \sub H} \sheaffunctor(M)_K \Big\rbrace ^H 
\end{align*}
where $\theta_H(m)$ sends $K \leqslant H$ to $m_K$, the image of $m \in M(H)$ in $\sheaffunctor(M)_K$.
We must check that $\theta_H(m)$ defines an $H$-equivariant map. 
As restriction and applying idempotents are equivariant, the square below commutes.
\[
\xymatrix{
M(H) \ar[r] \ar[d]^{C_h} &
\sheaffunctor(M)_K \ar[d]^{C_h} \\
M(H) \ar[r]  &
\sheaffunctor(M)_{hKh^{-1}} 
}
\]
Since $C_h= \id_{M(H)}$, we have the desired equivariance. 
We also see that $m_K$ is fixed by $N_H(K)$. 

We can restrict a section $s =\theta_{NK} (m)$ coming from $m \in \idem{NK}{N}{NK} M(NK)$ 
to a section over 
\[
O_{NK}(N,NK) = O_G(N,NK).
\] 
The sets
$s(O_G(N,NK))$ for varying open and normal $N$, 
closed $K$, and $s=\theta_{NK}(m)$ for $m$ in $\idem{NK}{N}{NK} M(NK)$,
define a topology on  $\sheaffunctor(M)$ by Lemma \ref{lem:sectionbasis}.

The $G$-action is continuous by Lemma \ref{lem:ctsGact}
and the projection map $p$ is a $G$-equivariant local homeomorphism
by Proposition \ref{prop:ctsGmaplocal}.

Lemma \ref{lem:Qmod} shows that $(\sheaffunctor(M),p)$ is a sheaf of $\bQ$-modules.
Lemma \ref{lem:Weyl} completes the construction by showing that we have a
Weyl-$G$-sheaf.
\end{construction}

\begin{lemma}\label{lem:sectionbasis}
The sets $s(O_G(N,NL))$ constructed above form a basis for a topology
on
\[
\sheaffunctor(M)= \underset{K\in \sub G}{\coprod} \sheaffunctor(M)_K
=
\underset{K \in \sub G}{\coprod} \underset{N \opensub G}{\colim} \idem{NK}{N}{NK}  M(NK).
\]
\end{lemma}
\begin{proof}
Given $a \in \sheaffunctor(M)_L$, we can find a representative $m \in \idem{NL}{N}{NL} M(NL)$.
We then take the composite
\[
\idem{NL}{N}{NL} M(NL) 
\lra 
M(NL) 
\xrightarrow{\theta_{NL}}
\Big\lbrace s \colon \sub  NL \lra \coprod_{K \in \sub NL} \sheaffunctor(M)_K \Big\rbrace ^{NL}.
\]
We see that $m_L =\theta_{NL}(m)(L)= a$ and so each germ is in the image of one of our chosen sections.
It follows that the open sets $s(O_G(N,NK))$ cover $\sheaffunctor(M)$.

Now we show that the intersection of two basis sets is
a union of basis elements. 
Take sections $t_1=\theta_{N_1 K_1} (m_1)$ and $t_2=\theta_{N_2 K_2} (m_2)$
and let
\[
x \in t_1(O_G(N_1,N_1K_1))\cap t_2(O_G(N_2,N_2K_2)).
\]
We construct a $t(O_G(N,NK))$ satisfying
\[
x \in t(O_G(N,NK)) \subseteq t_1(O_G(N_1,N_1K_1))\cap t_2(O_G(N_2,N_2K_2)).
\]
Given such an $x$, we let $L=p(x)$, which is a closed subgroup of $G$.  
We see that 
\[
L \in O_G(N_1,N_1K_1)\cap O_G(N_2,N_2K_2)
\]
and ${t_1}(L)={t_2}(L)$ in $\sheaffunctor(M)_L$.
Hence, there is an open normal subgroup $N \leqslant N_1 \cap N_2$ of $G$ such that
$m_1$ and $m_2$ restricted to $\idem{NL}{N}{NL} M(NL)$ agree.
It follows that $t_1$ and $t_2$ agree when restricted to 
$O_{NL}(N,NL) = O_G(N,NL)$.
We define $t$ to be this common refinement and see
\[
t(O_G(N,NL))\subseteq t_1(O_G(N_1,N_1K_1))\cap t_2(O_G(N_2,N_2K_2)). \qedhere
\]
\end{proof}

\begin{lemma}\label{lem:ctsGact}
The $G$-action on the space $\sheaffunctor(M)$ from Construction \ref{con:sheaf_mack} is continuous.
\end{lemma}
\begin{proof}4
Take a basic open set $s(O_G(N,NK))$ for $s=\theta_{NK}(m)$ with $m \in M(NK)$.
Take any point $(g,t_L)$ in the pre-image of $s(O_G(N,NK))$ under the group action map.
Then
\[
(gt)(gLg^{-1})=s(gLg^{-1}) \in \sheaffunctor(M)_{gLg^{-1}}.
\]
We can therefore find an open normal subgroup $N$ and representatives such that
\[
(gt)_{\mid_{O_G(N,NgLg^{-1})}}=s_{\mid_{O_G(N,NgLg^{-1})}},
\quad \textrm{ hence } \quad
t_{\mid_{O_G(N,NL)}}=(g^{-1}s)_{\mid_{O_G(N,NL)}}.
\]

These sections have $N$-invariant domains and are obtained from an $N$-fixed module $M(NL)$.
We can conclude that the open set:
\begin{align*}
W=gN \times t(O_G(N,NL))
\end{align*}
is contained in pre-image of $s(O_G(N,NK))$ under the group action map.
\end{proof}

\begin{proposition}\label{prop:ctsGmaplocal}
The projection map in Construction \ref{con:sheaf_mack} is a continuous $G$-map
and a local homeomorphism.
\end{proposition}
\begin{proof}
We first prove that $p$ is continuous. Given a basic open subset of $\sub G$ of the form $O_G(N,NK)$, we have that
\[
p^{-1}(O_G(N,NK))=\underset{L \in O_G(N,NK)}{\coprod} \sheaffunctor(M)_L.
\]
Take $s \in \sheaffunctor(M)_L$ and a section
\[
t \co O_G(N',N' L) \lra \underset{L \in O_G(N',N' L)}{\coprod} \sheaffunctor(M)_L.
\]
with $t(L) =s$  and $N' \leqslant N$. Then $t(O_G(N',N' L)) \subseteq p^{-1}(O_G(N,NK))$
and is open.

We now show that $p$ is a local homeomorphism.
A point $s$ in $\sheaffunctor(M)_L$ has a neighbourhood of the form $t(O_G(N,NK))$ as seen previously. 
Let $f=p_{|t(O_G(N,NK))}$, we claim that
\begin{align*}
f \co t(O_G(N,NK)) \lra O_G(N,NK)
\end{align*}
is a homeomorphism.
Since $f$ is bijective and continuous, we need only show that it is open.

A basic open set of $t(O_G(N,NK))$ is a set of the form
$t'(O_G(N',N'K')) \leqslant t(O_G(N,NK))$
for  $O_G(N',N'K') \subseteq O_G(N,NK)$.
The map $f$ sends $t'(O_G(N',N'K'))$ to $O_G(N',N'K')$, which is open in $O_G(N,NK)$.
\end{proof}

\begin{lemma}\label{lem:Qmod}
If $U$ is any open subset of $\sub G$, then the set of sections of $p$ over $U$
from Construction \ref{con:sheaf_mack} is a $\mathbb{Q}$-module.
Hence $(\sheaffunctor(M),p)$ is a sheaf of $\bQ$-modules.
\end{lemma}
\begin{proof}
Given two sections $s$ and $t$ of $p$ we can restrict ourselves to
neighbourhoods of $\sub G$ and $E$ such that $p$ is a local homeomorphism.
From the construction of the topology on $E$
it follows that on this region, $s$ and $t$ come from
elements of some $M(NK)$,
where $N$ is an open normal subgroup and $K$ is a closed subgroup of $G$.
The addition of $M(NK)$ defines a section $s+t$.
Similarly, the $\bQ$-action on $M(NK)$ defines $q \cdot t$ for $q \in \bQ$.
\end{proof}

\begin{lemma}\label{lem:Weyl}
The $G$-sheaf of $\bQ$-modules $(\sheaffunctor(M),p)$ defined in Construction \ref{con:sheaf_mack}
is a Weyl-$G$-sheaf.
That is, $g \in G$ gives maps of $\bQ$-modules
\[
p^{-1}(K) = \sheaffunctor(M)_K \lra \sheaffunctor(M)_{g K g^{-1}}
\]
and $\sheaffunctor(M)_K$ is $K$-fixed for each $K \in \sub G$.
\end{lemma}
\begin{proof}
If $K\in \sub G$, then
\[
p^{-1}(K) = \sheaffunctor(M)_K = \colim_J M(J)_{(K)}
\]
with $J$ running over all of the open subgroups containing $K$.
An element $g \in G$ induces maps
\[
\sheaffunctor(M)_K \lra \sheaffunctor(M)_{g K g^{-1}}
\]
by acting as $C_g$ on the terms $ M(J)_{(K)}$.
This action is a map of $\bQ$-modules and hence gives a
$\bQ$-module map on the stalks.

The $\bQ$-module $M(J)$ is $J$-fixed and the idempotents defining the
stalk $M(J)_{(K)}$ are $J$-fixed and hence $K$-fixed.
Thus $(\sheaffunctor(M),p)$ is a Weyl-$G$-sheaf.
\end{proof}

We summarise this work in the following theorem.
The additional statement here is that the construction is functorial.
This follows from the fact that maps of Mackey functors
commute with actions of Burnside rings.

\begin{theorem}\label{thm:macktosheaf}
For $G$ a profinite group, there is a functor
\begin{align*}
\sheaffunctor \colon \mackey{G} \lra \weylsheaf{G}
\end{align*}
which sends a Mackey functor $M$ over $G$ to a Weyl-$G$-sheaf denoted $\sheaffunctor(M)$,
as defined in Construction \ref{con:sheaf_mack}.
\end{theorem}

We end this subsection by noting that the maps $\theta_H$ from 
Construction \ref{con:sheaf_mack} are injective. 

\begin{lemma}\label{lem:injectivemaptostalks}
Let $m \in M(H)$. If $\theta_H(m)(L) =m_L =0$ for all $L \leqslant H$, then $m=0$ in $M(H)$. 
\end{lemma}
\begin{proof}
For $L \leqslant H$, there is an open normal subgroup $N$ such that 
the image of $m$ in $\idem{NL}{N}{NL} M(NL)$ is zero. 
As $m$ is $H$-fixed, the image of $m$ is $N_H(NL)$-fixed.
By Lemma \ref{lem:fixinflate}, restriction and applying idempotents gives an isomorphism
\[
\idem{H}{N}{NL}  M(H)
\cong
\idem{NL}{N}{NL}   M(NL)^{N_H(NL)}.
\]
Hence, we see that $\idem{H}{N}{NL} m$ is zero in $\idem{H}{N}{NL} M(H)$.
As the sets $\overline{O}_H(N,NL)$ are an open cover of the compact space $\sub H/H$, 
we see that $m$ must be zero in $M(H)$. 
\end{proof}

\section{The equivalence}\label{sec:equivalence}
In this section we prove that the two functors $\mackeyfunctor$ and $\sheaffunctor$ are inverse equivalences, see Theorem \ref{thm:equivalencemain}.

\subsection{The equivalence on sheaves}

We prove that for a Weyl-$G$-sheaf of $\bQ$-modules $F$,
we have an isomorphism of equivariant sheaves of $\bQ$-modules
\[
\sheaffunctor \circ \mackeyfunctor(F) \cong F.
\]
The first part is to show that we have an isomorphism on each stalk.
For that, we need to know how the action of the Burnside ring on Mackey functors
translates to sheaves.

\begin{lemma}\label{lem:burnsheafrest}
Let $F$ be $G$-sheaf of $\bQ$-modules on $\sub G$
and $K \leqslant H$ open subgroups of $G$.
The Mackey functor $M=\mackeyfunctor{F}$ satisfies
\[
\left(\left[H/K\right](s)\right)(L)=|(H/K)^L| s(L)
\]
for all $L \in \sub H$ and $s \in M(H) = F(\sub H)^H$.
\end{lemma}
\begin{proof}
The element $[H/K] \in \burnsidering(H)$ acts on $M(H)$ by the formula
\[
\left[H/K \right](s)=\underset{hK\in H/K}{\sum}C_h\overline{s_{|_{\sub K}}}.
\]
As $s \in M(H)$ is $H$-fixed,
the section $t=C_h\overline{s_{|_{\sub K}}}$ satisfies
$t(hJh^{-1})= s(hJh^{-1})$ for $J \in \sub K$.
It is zero outside of $\sub (hKh^{-1})$.
Hence, for $L \in \sub H$, $t(L)$ is non-zero exactly when
$L \in \sub (hKh^{-1})$. This is equivalent to the condition
$h NA \in (H/K)^L$.
We see that
\[
\left(\left[H/K\right](s)\right)(L)
=
\underset{hK\in G/K}{\sum}\overline{s_{|_{\sub (hKh^{-1})}}} (L)
=
\underset{hK\in (H/K)^L}{\sum}s(L). \qedhere
\]
\end{proof}

We can now see that a idempotent of $\burnsidering(H)$ acts by restricting a section 
and then extending the result by zero.  

\begin{proposition}\label{prop:idemactionsheaf}
For $H$ an open subgroup of $G$, let $U$ be an $H$-invariant open and closed subset of $\sub H$.
If $F$ is a $G$-sheaf then $e_U^H\mackeyfunctor(F)(H)$ is equal to the set
of $H$-equivariant sections of $\sub H$ which are zero outside $U$. 
Hence, $e_U^H\mackeyfunctor(F)(H) \cong F(U)^H$. 
\end{proposition}
\begin{proof}
The definition of $\mackeyfunctor(F)$ gives:
\begin{align*}
\mackeyfunctor(F)(H)=F(\sub H)^H.
\end{align*}
Since the sets of the form $O_H(N,NK)$ form a basis for $\sub H$ where $N$ is open
and normal in $H$, we can assume that $U=O_H(N,NK)$.
By Proposition \ref{prop:idempotents}
(and using the notation $\alpha_{NA,NK}$ from the proof of that proposition), 
we have the first equality below for $s \in F(\sub H)^H$.
\begin{align*}
\idem{H}{N}{NK} s(L)
&=
\underset{N \leqslant NA\leqslant NK}{\sum}
\left(
\alpha_{NA,NK} [G/NA] (s)
\right)(L) \\
&=
\Big(\underset{N \leqslant NA\leqslant NK}{\sum} 
\alpha_{NA,NK} |(G/NA)^L| \Big)s(L) \\
&=\idem{H}{N}{NK}(L) s(L)
\end{align*}
The second is Lemma \ref{lem:burnsheafrest}.
The last is an instance of how $[G/NA] \in \burnsidering(H)$
defines a function from $\sub H/H$ to $\bQ$, see Theorem \ref{thm:burnchar}.
\end{proof}

\begin{proposition}\label{prop:stalk}
If $F$ is a Weyl-$G$-Sheaf over $\sub G$, then for each $K\in \sub G$ we have
an isomorphism of $\bQ$-modules
\[
\psi_K \co \sheaffunctor \circ\mackeyfunctor(F)_K \lra F_K.
\]
\end{proposition}
\begin{proof}
The left hand side can be expanded to 
\begin{align*}
\underset{N \opensub G}{\colim} \idem{NK}{N}{NK} \mackeyfunctor(F)(NK)
&=
\underset{N \opensub G}{\colim} \idem{NK}{N}{NK} F(\sub (NK))^{NK}  \\
& \cong 
\underset{N \opensub G}{\colim}  F(O_{NK}(N, NK))^{NK}
\end{align*}
using Proposition \ref{prop:idemactionsheaf} for the last term.
The right hand side is 
\[
\underset{N \opensub G}{\colim} F(O_{NK}(N, NK)). 
\]
The inclusions $F(O_{NK}(N, NK)^{NK} \to F(O_{NK}(N, NK)$
induce the desired map $\psi_K$.
It is an isomorphism as any germ can be represented by 
a section that is locally sub-equivariant by Proposition~\ref{prop:Weylequi}
\end{proof}

The proof requires $F$ to be a Weyl-$G$-sheaf as we need the
local sub-equivariance property.

We now check that the preceding isomorphism is compatible with the group actions.
\begin{lemma}\label{lem:equivequi}
Let $F$ be a Weyl-$G$-sheaf of $\bQ$-modules.
If $K\in \sub G$ and $g\in G$, then the following square commutes
\begin{align*}
\xymatrix@C+1cm{
\sheaffunctor \circ\mackeyfunctor(F)_K
\ar[r]^-{\psi_K}
\ar[d]^-{C_g}
&
F_K
\ar[d]^-{C_g} \\
\sheaffunctor \circ\mackeyfunctor(F)_{gKg^{-1}}
\ar[r]^-{\psi_{gKg^{-1}}}
&
F_{gKg^{-1}}
}
\end{align*}
where $\psi_K$ is the map defined in the proof of Proposition \ref{prop:stalk}.
\end{lemma}
\begin{proof}
A germ $s_K\in \sheaffunctor \circ\mackeyfunctor(F)_K$ 
can be represented by an $NK$-equivariant section 
\[
s \co O_{NK}(N, NK) \to F.
\]
This section is also a representative for $\psi_K(s_K)$, from which
the commutativity of the square follows. 
\end{proof}

\begin{theorem}\label{thm:weylmack}
If $F$ is any Weyl-$G$-sheaf of $\mathbb{Q}$-modules, then the maps $\psi_K$ induce an isomorphism
\begin{align*}
\psi \co \sheaffunctor \circ \mackeyfunctor(F) \lra F
\end{align*}
in the category of Weyl-$G$-Sheaves.
\end{theorem}
\begin{proof}
By Proposition \ref{prop:stalk} and Lemma \ref{lem:equivequi}, we see that they are isomorphic as $G$-sets.
To prove that they are topologically equivalent, we need to show that they have the same sections.
As both objects are sheaves, it suffices to do so locally, which amounts to considering
sections that represent stalks.
By Proposition \ref{prop:stalk}, a germ over $K \in \sub G$
of $\sheaffunctor \circ \mackeyfunctor(F)$ is represented by 
some element of $F(O_{NK}(N, NK) )^{NK}$.
By the local sub-equivariance property of Proposition~\ref{prop:Weylequi}
the same is true for a stalk of $F$.
\end{proof}

\subsection{The equivalence on Mackey functors} The starting point is to give a levelwise isomorphism of 
$\bQ$-modules.

\begin{proposition}\label{prop:homeo}
If $H$ is an open subgroup of a profinite group $G$ and $M$ a Mackey functor for $G$, then there
is an isomorphism of $\bQ$-modules (constructed in the proof)
\[
\theta_H \co M(H) \lra \mackeyfunctor \circ \sheaffunctor(M)(H).
\]
\end{proposition}
\begin{proof}
Recall from Constructions \ref{con:mackey_construct} and \ref{con:sheaf_mack}
that the stalks of $\sheaffunctor(M)$ are given by
\[
\sheaffunctor(M)_K = \underset{J \closedsuper K}{\colim} M(J)_{(K)} =
\underset{N \opensub G}{\colim} M(NK)_{(K)} =
\underset{N \opensub G}{\colim} \idem{NK}{N}{NK}  M(NK).
\]
Recall the map $\theta_H$ from Construction \ref{con:sheaf_mack}
\[
\theta_H \colon M(H) 
\lra 
\Big\lbrace s \colon \sub  H\lra \coprod_{K \in \sub H} \sheaffunctor(M)_K \Big\rbrace ^H 
\]
where $\theta_H(m)(K) = m_K$, the image of $m$ in $\sheaffunctor(M)_K$. The
topology on $\sheaffunctor(M)$ is defined in terms of $\theta_H(m)$ for varying $H$ and $m$.

The composite of the two functors at $H$ is
\begin{align*}
\mackeyfunctor \circ \sheaffunctor(M)(H)  &=
\sheaffunctor(M)(\sub H)^H \\
&=
\Big\lbrace s \colon \sub  H\lra \coprod_{K \in \sub H} \sheaffunctor(M)_K \Bigm| \textrm{$s$ continuous},\,p\circ s=\id \Big\rbrace^H. 
\end{align*}
Since the topology was defined using the images of $\theta_H$, it follows 
that $\theta_H(m)$ is a continuous section. 
Thus $\theta_H$ defines a map as in the statement. 
Moreover, it is additive and surjective, 
injectivity follows from Lemma \ref{lem:injectivemaptostalks}.
\end{proof}

\begin{theorem}\label{thm:mackeyweyl}
The maps $\theta_H$ induce a natural isomorphism of Mackey functors
\[
\theta \colon M \lra \mackeyfunctor \circ\sheaffunctor(M).
\]
\end{theorem}
\begin{proof}
We need to show that the correspondence between $M$ and $\mackeyfunctor \circ\sheaffunctor(M)$ commutes with the three maps; restriction, induction and conjugation. 

For conjugation and restriction one can calculate the effect directly on a stalk of some $m \in M(H)$, 
for $H$ an open subgroup of $G$.
One will see that the map $\theta$ is compatible with conjugation as the restriction maps of 
a Mackey functor are equivariant and the action of the Burnside ring is compatible with conjugation. 
Similarly, since $\theta$ is defined via the restriction maps of Mackey functors and the restriction of
$\mackeyfunctor \circ\sheaffunctor(M)$ is defined by restricting a section to a smaller domain,
$\theta$ is compatible with restrictions.

Induction requires more work. Suppose $K\leqslant H$ are open and $s\in M(K)$. We must show 
that the square 
\begin{align*}
\xymatrix{
M(H)
\ar[r]_(0.3){\theta_H}
&\mackeyfunctor \circ \sheaffunctor (M)(H)\\
M(K) 
\ar[u]^{I^H_K}
\ar[r]_-{\theta_{K}}
&\mackeyfunctor \circ \sheaffunctor(M)(K)
\ar[u]^{\overline{I^H_K}}}
\end{align*}
commutes, where $\overline{I^H_K}$ is the induction map for $\mackeyfunctor \circ \sheaffunctor(M)$. 

Chasing through the definitions gives 
\begin{align*}
\left( \overline{I^H_K} \theta_K(m) \right)_L
=
\Big(\underset{h \in H/K}{\sum} h\ast \overline{ \theta_K(m) }\Big)_L 
=
\sum_{\substack{h \in H/K \\ h^{-1} L h \leqslant K}} (h \ast \theta_K(m))_L.
\end{align*}
In the final equality we replace the extension by zero by an explicit condition
on the subgroups.  

Conversely, we use the Mackey axiom to rewrite $(\theta_H I^H_K(m))_L \in \sheaffunctor(M)_L$. 
We may choose a representative $a \in \idem{NL}{N}{NL} M(NL)$ 
for this germ, with $N \leqslant K$ an open normal subgroup of $G$.  
Then 
\begin{align*}
a 
& = 
\idem{NL}{N}{NL}  R^H_{NL} I^H_K(m) \\
& = 
\idem{NL}{N}{NL}  \underset{\left[ NL \backslash H\slash K \right]}{\sum} 
I^{NL}_{NL\cap xKx^{-1}} \circ  C_x \circ  R^K_{K\cap x^{-1}NLx} (m) \\
&=
\idem{NL}{N}{NL}  \underset{\left[ NL\backslash H\slash K \right]}{\sum} 
I^{NL}_{NL\cap xKx^{-1}}\circ R^{xKx^{-1}}_{NL \cap xKx^{-1}}\circ C_x (m). \\
&=
\underset{\left[ NL\backslash H\slash K \right]}{\sum} 
I^{NL}_{NL\cap xKx^{-1}} \Big( 
R^{NL}_{NL \cap xKx^{-1}} \big( \idem{NL}{N}{NL} \big) \cdot R^{xKx^{-1}}_{NL \cap xKx^{-1}} C_x (m) \Big). 
\end{align*}
Consider the term $R^{NL}_{NL \cap xKx^{-1}} \idem{NL}{N}{NL}$. 
If this is non-zero, there is a subgroup $J$ of $NL \cap xKx^{-1}$
in the support of the idempotent. Hence,
\[
L \leqslant NL =NJ \leqslant xKx^{-1}.
\]
If $x \in \left[ NL\backslash H\slash K \right] = \left[ L\backslash H\slash K \right]$ is such that 
$L\leqslant xKx^{-1}$, then $NL \cap xKx^{-1} =NL$.
The corresponding summand is
\begin{align*}
\idem{NL}{N}{NL} \cdot R^{xKx^{-1}}_{NL}  C_x (m)
\end{align*}
which is sent to $C_x (m)_L$ in the $L$-stalk. 

The collection of $x \in \left[ L\backslash H\slash K \right]$ such that $L\leqslant xKx^{-1}$,
is exactly the set $hK \in H/K$ such that $h^{-1} L h \leqslant K$
(the canonical isomorphism sends $LxK$ to $xK$).
Hence,
\[
(\theta_H I^H_K(m))_L 
= 
\sum_{\substack{h \in H/K \\ h^{-1} L h \leqslant K}} (\theta_{hKh^{-1}} C_h(m))_L  
=
\sum_{\substack{h \in H/K \\ h^{-1} L h \leqslant K}} (h \ast \theta_{K} (m))_L. \qedhere
\]
\end{proof}

We combine Theorems \ref{thm:weylmack} and \ref{thm:mackeyweyl} and note that our constructions are 
compatible with maps in the two categories to obtain the main result.

\begin{theorem}\label{thm:equivalencemain}
If $G$ is a profinite group then the category of rational $G$-Mackey functors is equivalent to the category of Weyl-$G$-sheaves over $\sub G$. 
Furthermore, this is an exact equivalence.
\end{theorem}

\section{Consequences}\label{sec:consequences}

\subsection{Examples}

Just as non-equivariantly, one can make constant sheaves and skyscraper sheaves.
We leave the details to the second author's thesis \cite{sugruethesis}
and the forthcoming paper by the authors \cite{BSsheaves}. 
The stalks of the constant sheaf at a $\bQ$-module $B$ are all isomorphic to $B$.
The sections are as for the non-equivariant constant sheaf.

\begin{example}
The Burnside ring Mackey functor corresponds to the constant Weyl-$G$-Sheaf at $\bQ$. 
The constant sheaf at a $\bQ$-module $B$ corresponds to the Mackey functor $\burnsidering \otimes B$.
\end{example}

For an equivariant skyscraper sheaf, one picks a closed subgroup $K$ of $G$
and a discrete $W_G K$-module $C$. 
The stalk at $K$ is $C$, for $g \in G$, the stalk at $gKg^{-1}$ is $gC$ 
(with the expected $gKg^{-1}$-action), all other stalks are zero. 
The sections of this sheaf over $U$ are given by a direct sum of copies of $gB$, 
with one copy for each $gKg^{-1} \in U$, 
where the sums runs over $G/K$. 

\begin{example}
For $K$ a closed subgroup of $G$, we can describe the Mackey functor $M$
that corresponds to the skyscraper sheaf $\sky_K(C)$ 
built from a discrete $W_G K$-module $C$.
At an open subgroup $H$, $M(H)$ is the $H$-fixed points of $\sky_K(C)(\sub H)$. 
When $K$ is open, there are only finitely many distinct conjugates of $K$, 
hence we can express $M(H)$ as a finite direct sum
indexed by the distinct conjugates $gKg^{-1}$ of $K$ that are subgroups of $H$
\[
M(H)=\sky_K(C)(\sub H)^H = \Big( \bigoplus_{ \{gKg^{-1} \leqslant H\} } gC  \Big)^H.
\]
If $K$ is the trivial group, then the Mackey functor $M_C$ corresponding to $\sky_{\{ e\}}(C)$ is 
the fixed point Mackey functor, see Example \ref{ex:fixpointmackey}, 
of the discrete $G$-module $C$. 

It can be illuminating to perform the reverse calculation. 
The fixed point Mackey functor $M_C$ is cohomological, 
that is, restriction followed by induction is multiplication by the subgroup index
\[
I_K^H R_K^H = [H,K].
\]
See Th\'{e}venaz and Webb \cite[Section 16]{tw95} for further details of cohomological Mackey functors.

This equation simplifies the behaviour of idempotents, particularly when we calculate the stalk of
$\sheaffunctor(M_C)$ at some closed subgroup $K$. 
As a self-map of $M_C(NK)$,
\begin{align*}
\idem{NK}{N}{NK} 
& =\sum_{N \leqslant NA\leqslant NK}\frac{|NA|}{|N_NK(NK)|}\mu(NA,NK) [NK:NA] \\
& =\sum_{N \leqslant NA\leqslant NK}\frac{|NA|}{|NK|}\mu(NA,NK) [NK:NA] \\
& =\sum_{N \leqslant NA\leqslant NK} \mu(NA,NK) 
\end{align*}
which is zero for non-trivial $K$ and the identity map for $K$ the trivial group. 
Hence, the stalks at non-trivial subgroups are zero. 
We then use the fact that $C$ is discrete to complete the calculation 
\[
C = \underset{N \opensub G}{\colim} C^N  = \underset{N \opensub G}{\colim} M_C(N).
\]
\end{example}

\subsection{Weyl Sheaves from equivariant sheaves}\label{subsec:weyltosheaf}

As mentioned, Construction \ref{con:mackey_construct} does not require the input
$G$-sheaf to be a Weyl-$G$-sheaf. Hence, given a sheaf $F$, 
$\sheaffunctor \circ \mackeyfunctor (F)$ is a Weyl-$G$-sheaf. 
We can construct this operation explicitly. 

\begin{proposition}
If $F$ is a $G$-sheaf then the underlying set of 
$\sheaffunctor \circ \mackeyfunctor (F)$ is 
\begin{align*}
\underset{K\in \sub G}{\coprod}F_K^K.
\end{align*}
Furthermore, the map
$F \to \sheaffunctor \circ \mackeyfunctor (F)$
is induced from the stalkwise inclusions of fixed points. 
\end{proposition}
\begin{proof}
The stalks are given by 
\[
\underset{N\opennormalsub G}{\colim}\,e^{NK}_{O(N,NK)} F(\sub NK)^{NK} \cong F_K^K.   \qedhere
\]
\end{proof}
One can prove directly that the stalks $F_K^K$ form a Weyl-$G$-sheaf, 
as Proposition \ref{prop:Weylequi} implies that a $K$-fixed germ can be represented 
by a section that is $N K$-equivariant for some open normal subgroup $N$.
One can also prove that a map from a Weyl-$G$-sheaf $E$ to $F$ will factor through 
the composite $\sheaffunctor \circ \mackeyfunctor (F)$. 
See \cite[Section 10]{BSsheaves} for further such results.

\bibliographystyle{alpha}
\bibliography{ourbib}

\end{document}